\documentclass[11pt]{amsart}

\usepackage{amsmath}
\usepackage{amssymb}
\usepackage{latexsym}
\usepackage{amscd}
\usepackage[frame,cmtip,arrow,matrix,line,graph,curve]{xy}

\newtheorem{theorem}{Theorem}[section]
\newtheorem{proposition}[theorem]{Proposition}

\newtheorem{lemma}[theorem]{Lemma}

\newtheorem{remit}[theorem]{Remark}
\newtheorem{definition}[theorem]{Definition}

\newcommand{\pp}{\mathbb{P}}

\newcommand{\cc}{\mathbb{C}}

\newcommand{\TT}{\mathbb{T}}

\newcommand{\rk}{\mathrm{rk}}

\newcommand{\cF}{\mathcal{F} }

\newcommand{\cO}{\mathcal{O} }
\newcommand{\cP}{\mathcal{P} }

\begin{document}

\title[Secant varieties and Hirschowitz bound]{Secant varieties and Hirschowitz bound \\on vector
bundles over a curve}
\date{}
\author{Insong Choe and George H. Hitching}
\address{Department of Mathematics, Konkuk University, 1 Hwayang-dong, Gwangjin-Gu, Seoul, Korea}
\email{ischoe@konkuk.ac.kr}
\address{H\o gskolen i Vestfold, Boks 2243, N-3103 T\o nsberg, Norway} \email{george.h.hitching@hive.no}

\thanks{This work was supported by Korea Research Foundation Grant funded by the Korean Government
(KRF-2008-331-C00033).} 


\begin{abstract}
For a vector bundle $V$ over a curve $X$ of rank $n$ and for each
integer $r$ in the range $1 \le r \le n-1$, the Segre invariant
$s_r$ is defined by generalizing the minimal self-intersection
number of the sections on a ruled surface. In this paper we
generalize Lange and Narasimhan's results on rank 2 bundles which
related the invariant $s_1$ to the secant varieties of the curve
inside certain extension spaces. For any $n$ and $r$, we find a way
to get information on the invariant $s_r$ from the secant varieties
of certain subvariety of a scroll over $X$. Using this geometric
picture, we obtain a new proof of the Hirschowitz bound on $s_r$.
\end{abstract}
\maketitle

\section{Introduction}

Let $X$ be a smooth algebraic curve of genus $g \ge 2$. Let $V$ be a
vector bundle over $X$ of rank $n$. For each $1 \le r < n$, the {\it
$r$-th Segre invariant of $V$} is defined by
\[
%s_r(V) := \min_{E \subset V} \{ r \deg(V) - n \deg (E) \},
s_r(V) := \min \{ r \deg(V) - n \deg (E) \},
\]
where the minimum is taken over the subbundles $E$ of $V$ of rank
$r$. If a subbundle $E$ realizes this minimum, that is, $s_r(V) = r
\deg(V) - n \deg (E)$, then $E$ is called a {\it maximal subbundle
of $V$}.

When $n = 2$ and $r=1$, we have $s_1(V)= -e$ for the classical Segre
invariant $e$ of the ruled surface $\pp V$ defined by the minimal
self-intersection number of the sections (\cite{Ha}, V \S2). In
general, the invariant $s_r(V)$ has an alternative definition in
terms of the intersection numbers (see \cite{La3}).

The invariant $s_r$ is lower semicontinuous: for any family $\{V_t:
t \in T \}$ of rank $n$ bundles parameterized by a variety $T$, the
sublocus
\[ \{t \in T : s_r(V_t) \le s \} \]
is closed for each $s$. Hence $s_r$ induces a natural stratification
on the moduli space $U(n,d)$ of vector bundles over $X$ of rank $n$
and degree $d$. This stratification has been studied by several
authors (\cite{BrLa}, \cite{La1}, \cite{RuTe}).

Note that $V$ is semistable if and only if $s_r(V) \ge 0$ for each
$r$ such that $1 \le r < n$. There are several results on the upper
bound of $s_r$. Nagata's bound on ruled surfaces \cite{Na} says that
$s_1(V) \le g$ when $n=2$ and $r=1$. For arbitrary $n$ and $r$,
Mukai and Sakai \cite{MuSa} proved that $s_r(V) \le r(n-r)g$. The
sharp upper bound of $s_r$ was obtained by Hirschowitz:
\begin{proposition} \label{Hirschowitz} \  {\rm (\cite{Hi}, Th\'{e}or\`{e}me 4.4)} \ \
Let $V$ be a vector bundle of rank $n$ over $X$ of genus $g$. For $1
\le r < n$, we have
\begin{equation} \label{bound}
s_r(V) \le r(n-r)(g-1) + \varepsilon,
\end{equation}
where $\varepsilon$ is the integer satisfying
\[ 0 \le \varepsilon \le n-1 \quad \hbox{and} \quad r(n-r)(g-1) + \varepsilon \equiv r \deg (V) \mod n. \qed \]
\end{proposition}
\indent As the main result of this paper, we reprove this bound by
relating the invariants $s_r$ to the geometry of certain secant
varieties. This yields generalizations of the results of Lange and
Narasimhan for bundles of rank 2 (\cite{LaNa}, \cite{La2}). Let us
briefly review their results here.

One can show that every vector bundle $V$ of rank 2 and degree $d
\gg 0$ fits into the exact sequence
\[
0 \to \cO_X \to V \to L \to 0
\]
for some line bundle $L$. So the bundle $V$ corresponds to a point
\[ v \in \pp_L := \pp H^1(X, L^{-1}) \cong \pp H^0(X, K_X L)^\vee . \]
The curve $X$ maps into $\pp_L$ via the linear system $|K_X L|$. We
consider the secant variety $Sec^k X$ which is the closure of the
union of $\pp^{k-1} \subset \pp_L$ spanned by $k$ general points of
$X$. The invariant $s_1$ and the secant variety $Sec^k X$ are
related as follows.
\begin{proposition} \label{LN} {\rm (\cite{LaNa}, Proposition 1.1)} \ \
Suppose that $s \equiv d \mod 2$ and $4-d \le s \le d$. Then $s_1(V)
\ge s$ if and only if $v \notin Sec^k X$, where $k = \frac{1}{2}
(d+s-2)$. \qed
\end{proposition}
By using this, Lange \cite{La2} reproved Nagata's bound $s_1(V) \le
g$: By Riemann-Roch, $\dim \pp H^1(X, L^{-1}) = d+g-2$ for $d \gg
0$. Since the curve $X$ has no secant defect, $\dim Sec^k X = 2k-1$.
Hence every $v \in \pp_L$ lies on $Sec^m X$ for $m = \left\lceil
\frac{d+g-1}{2} \right\rceil$. By Proposition \ref{LN}, then,
\[
s_1(V) \le 2m-d = 2 \left\lceil \frac{d+g-1}{2} \right\rceil - d \le
g.
\]
\indent In this paper, we reprove the Hirschowitz bound by
establishing a statement which generalizes Proposition \ref{LN} to
the case of rank $r$ subbundles of rank $n$ bundles for any pair $r$
and $n$ (Theorem \ref{geometry}). The key observation is that the
locus of rank 1 vectors inside a certain scroll over $X$ plays the
role of the curve $X$ in the rank 2 case. This idea recently
appeared in the second author's work on symplectic bundles of rank 4
over a curve. Since he studied the case of genus 2, the involved
data was secant lines (\cite{Hit}, Lemma 8). Here we fully
investigate the idea for curves of higher genus and higher secant
varieties.

In Section 2, we arrange a basic setup for our study. In particular,
the locus $\Delta$ of rank 1 vectors inside the scroll is defined.

In Section 3, we briefly recall the dictionary for interpreting the
secants of the scrolls in terms of the data of elementary
transformations.

In Section 4, the key technical result on the lifting of elementary
transformations is proved. The criterion on the lifting is given in
terms of the secant varieties of $\Delta$.

In Section 5, we apply this criterion to get the Hirschowitz bound.
Once we have the criterion on the lifting, it is straightforward to
get the expected bound on $s_r$. The only technical point is to show
that the variety $\Delta$ has no secant defect. This requires some
argument appealing to the Terracini lemma. The proof of the secant
non-defectiveness of $\Delta$ is completed by applying Hirschowitz'
lemma saying that the tensor product of two general bundles is
nonspecial.

\section{The rank 1 locus}
Let $X$ be a smooth algebraic curve of genus $g \ge 2$. Let $V$ be a
vector bundle over $X$ of rank $n >1$. Firstly, we establish a fact
which justifies the validity of the forthcoming discussion on
extension spaces.
\begin{lemma} \label{setting}
\begin{enumerate}
\renewcommand{\labelenumi}{(\roman{enumi})}
\item Suppose that $V$ is a general stable bundle in $U(n, d)$ where
$d> (2g-1)n$. For every positive integer $r$ with $1 \le r < n$, the
bundle $V$ fits into an exact sequence
\[
0 \to \cO_X^{\oplus (n-r)} \to V \to F \to 0
\]
for some $F \in U(r, d)$.
\item Under the same assumption on $n$, $r$ and $d$, a general stable
bundle $V$ is fitted into an exact sequence
\begin{equation} \label{basicext}
0 \to E \to V \to F \to 0
\end{equation}
for a general $E \in U(n-r,0)$ and a general $F \in U(r,d)$.
\end{enumerate}
\end{lemma}
\begin{proof}
(The following argument was kindly informed to us by Peter
Newstead.)

The case when $r=1$ was dealt with in the paper of Atiyah (\cite{At}
Theorem 3, where it is attributed to J.-P.\ Serre): Under the
assumptions, we see that $h^0(X, V) - h^0(X, V(-x)) = n$ for each $x
\in X$. Hence $V$ is generated by global sections and there is a
surjection $H^0(X, V) \otimes \cO_X \to V$. For each $x \in X$, let
$N_x$ be the kernel of the the evaluation map $H^0(X, V) \to V_x$.
The union $\displaystyle N = \bigcup_{x \in X} N_x $ has dimension
at most $h^0(X, V) - n +1$, so there is an $(n-1)$-dimensional
subspace $P \subset H^0(X, V)$ such that $ P \cap N = \{ 0 \}$. This
means that there is an exact sequence
\[
0 \to P \otimes \cO_X \to V \to \det V \to 0.
\]
\indent For $r >1$, the same argument yields an exact sequence
\[
0 \to \cO_X^{\oplus  (n-r)} \to V \to F \to 0
\]
for some $F$ of rank $r$. Here the stability of $F$ is not
guaranteed, but $H^0(X, F^*) = 0$ by the stability of $V$. It is
well-known that the bundles $F$ of fixed rank and degree satisfying
$H^0(X, F^*) = 0$ form a bounded family. To get an irreducible
family, apply the above Serre--Atiyah construction to those $F$:
Consider the irreducible variety $\mathcal{T}$ parameterizing all
the extensions of the form
\[
0 \to \cO_X^{\oplus (r-1)} \to F \to L \to 0
\]
where $L \in Pic^d(X)$. From the above construction, every $F$
satisfying $H^0(X, F^*) = 0$, in particular every stable bundle $F
\in U(r,d)$, fits into some exact sequence in $\mathcal{T}$. Now
consider the family of extensions
\[
0 \to \cO_X^{\oplus (n-r)} \to V \to F \to 0
\]
for all $F$ fitting into some exact sequence in $\mathcal{T}$. This
is again parameterized by an irreducible family and contains all
stable bundles $V \in U(n,d)$. The proof of (i) is completed by
noting that $F$ is stable for the general member of the family.

Statement (ii) follows from (i) by deforming $\cO_X^{\oplus(n-r)}$
in $U(n-r,0)$ and $F$ in $U(r,d)$ respectively.
\end{proof}
Now we study the space $\pp := \pp H^1( X, Hom(F, E))$ associated to
the exact sequence (\ref{basicext}).
%\begin{convention} \label{convention} \ \
%$E \in U(n-r,0), \ \ F \in U(r,d)$ \ where \ $d > (2g-1)r.$ \qed
%\end{convention}
We consider the ruled variety $\pi: \pp Hom(F,E) \to X$. Let
$\cO(1)$ be the line bundle on $\pp Hom(F,E)$ such that $\pi_*
\cO(1) \cong Hom(F,E)^*$. By Serre duality and the projection
formula,
\begin{eqnarray*}
\pp &\cong& \pp H^0(X,  K_X \otimes Hom(F,E)^*)^\vee \\
&\cong& \pp H^0 (\pp Hom(F,E), \pi^*K_X \otimes \cO(1))^\vee.
\end{eqnarray*}
Hence there is a rational map
\[
\phi: \pp Hom(F,E) \dashrightarrow \pp
\]
given by the complete linear system $|\pi^*K_X \otimes \cO(1)|$.
\begin{lemma} \label{embedding} \ \
Let $E \in U(n-r,0)$ and $F \in U(r,d)$ where $d > (2g-1)r.$ Then
$\phi$ is an embedding.
\end{lemma}
\begin{proof}
It is easy to see that $\phi$ is an embedding if
\[
h^0(X, K_X \otimes Hom(F,E)^*)  - h^0(X, K_X(-x-y) \otimes
Hom(F,E)^*) = 2 r(n-r)
\]
for any $x,y \in X$. Via Serre duality, this holds if
\[
h^0(X, Hom(F,E)) = h^0 (X, Hom(F,E) \otimes \cO_X(x+y)) = 0.
\]
By our assumptions, $Hom(F,E)$ is semistable of slope $-d/r$, which
is smaller than $-2$.  Therefore we get the above cohomology
vanishing.
\end{proof}
Now we introduce the key object of our discussion: the locus
$\Delta$ defined by rank 1 vectors, or decomposable vectors, in the
scroll $\pp Hom(F,E)$.
\begin{definition} \label{rank1}
{\rm The locus $\Delta$ is the subvariety of $\pp Hom(F,E) $ defined
by $\Delta = \bigcup_{x \in X} \Delta_x$, where}
\[
\Delta_x = \{ [\mu \otimes e] \ : \ \mu \in F_x^*, \ e \in E_x \} \
\subset \ \pp Hom(F,E)|_x \cong \pp (F_x^* \otimes E_x).
\]
\end{definition}
\begin{remit}
It is easily seen that $\dim \Delta = n-1$. Also $\Delta \cong \pp
F^*$ if $r=n-1$. So the locus $\Delta$ reduces to the curve $X$ when
we consider line subbundles of rank 2 vector bundles.
\end{remit}

\section{Scrolls and elementary transformations}

The map $\phi: \pp Hom(F,E) \to \pp$ in the previous section can be
understood in terms of elementary transformations. In this section,
we briefly recall this connection.

Let $W$ be a vector bundle over $X$. We consider two kinds of
elementary transformations of $W$. Firstly, for any nonzero $\mu \in
W_x^*$, we get an exact sequence
\begin{equation} \label{associated}
0 \to \widetilde{W} \to W \to \cc_x \to 0
\end{equation}
whose restriction to $x$ gives
\[
0 \to \cc \to \widetilde{W}_x \to W_x \stackrel{\mu}{\to} \cc \to 0.
\]
The locally free sheaf or vector bundle $\widetilde{W}$ is called
the {\it elementary transformation of $W$ associated to $\mu$}.

Next, for any nonzero $w \in W_x$, there is a unique extension
\begin{equation} \label{at}
0 \to W \to \widehat{W} \to \cc_x \to 0
\end{equation}
which restricts to $x$ as
\[
0 \to \cc \cdot w  \to W_x \to \widehat{W}_x \to \cc \to 0.
\]
We call the vector bundle $\widehat{W}$ the {\it elementary
transformation of $W$ at $w$}.

These two processes are dual to each other: taking the dual of
(\ref{associated}), we get
\[
0 \to W^* \to (\widetilde{W})^* \to \cc_x \to 0.
\]
Here $(\widetilde{W})^*$ is nothing but the elementary
transformation of $W^*$ at $\mu$: under the above notations,
\[
\left( \widetilde{W} \right)^* \cong \widehat{(W^*)}.
\]
Now consider the map
\[
\phi: \pp W \to \pp H^1(X, W)
\]
given by the linear system $|\pi^*K_X \otimes \cO(1)|$, where $\pi:
\pp W \to X$ is the projection and $\cO(1)$ is the bundle on $\pp W$
satisfying $\pi_* \cO(1) \cong W^*$. As we noticed in Lemma
\ref{embedding}, if $W$ is semistable of slope less than $-2$, then
$\phi$ embeds $\pp W$ into $\pp H^1(X, W)$ as a scroll.

This embedding can be understood in terms of the elementary
transformations as follows. For each $[w] \in \pp W$, we consider
the elementary transformation $\widehat{W}$ of $W$ at $w$. From the
long exact sequence associated to (\ref{at}), we have the
1-dimensional kernel of the surjection $H^1(X, W) \to H^1(X,
\widehat{W})$  and this gives the image point $\phi ([w]) \in \pp
H^1(X, W)$. Accordingly, the secant line joining two points
$\phi([w_1])$ and $\phi([w_2])$ of $\pp W$ is given by
\[
\overline{[w_1][w_2]} \ = \ \pp \ker [H^1(X, W) \to H^1(X,
\widehat{W})]
\]
where $\widehat{W}$ is the elementary transformation of $W$ at $w_1$
{\bf and} $w_2$.

We will also need the description of the embedded tangent spaces of
$\pp W$. For each nonzero $w \in W_x$, we consider the following
diagram
\[
\begin{CD}
@. @. \widehat{W} \otimes \cO_X(x)|_x @= (\cc_x)^{\rk W} \\
@. @. @AAA @AAA @.\\
0 @>>> W  @>>>  \widehat{W} \otimes \cO_X(x) @>>>  \tau   @>>>  0   \\
@.   @|     @AAA         @AAA               @.  \\
0  @>>>    W                @>>>  \widehat{W}   @>>> \cc_x @>>> 0
\end{CD}
\]
The $\left( \rk (W) + 1 \right)$-dimensional dimensional kernel of
the map
\[
H^1(X,W) \to H^1(X, \widehat{W} \otimes \cO_X(x))
\]
gives the embedded tangent space $\TT_{[w]} \pp W$ in $\pp H^1(X,
W)$. This can be seen from the blowing-up and blowing-down
description of the elementary transformations. For details, we refer
the reader to \cite{Ma} and \cite{Tj}.

\section{A criterion for lifting of elementary transformations}
Consider an exact sequence
\begin{equation} \label{original}
0 \to E \to V \to F \to 0,
\end{equation}
where $E \in U(n-r,0)$ and $F \in U(r,d)$. This corresponds to a
point $v \in \pp Hom(F,E) = \pp$. In this section, we find a
geometric criterion to determine when an elementary transformation
$\widetilde{F}$ of $F$ has a lifting $\widetilde{F} \to V$ such that
the following diagram commutes.
\begin{equation} \label{lifting}
\xymatrix{ 0 \ar[r] & E \ar[r] & V \ar[r]  & F \ar[r] & 0 \\
& & & \widetilde{F} \ar[ul] \ar[u]}
\end{equation}
%\begin{equation} \label{lifting}
%\begin{CD}
%0 @>>> E @>>> V  @>>> F  @>>> 0 \\
%@. @. @.  @AAA @. \\
%@. @. @.  \widetilde{F} @.
%\end{CD}
%\end{equation}
Let $\widetilde{F}$ be an elementary transformation of $F$ given by
the sequence
\begin{equation} \label{elem}
0 \to \widetilde{F} \to F \to \tau \to 0
\end{equation}
for some torsion sheaf $\tau$ of degree $k$. Dualizing this, we get
\begin{equation} \label{dual}
0 \to F^* \to \widetilde{F}^* \to \tau' \to 0.
\end{equation}
Tensoring by $E$, we obtain
\begin{equation} \label{Elem}
0 \to Hom(F,E) \to  Hom(\widetilde{F},E) \to (\tau')^{\oplus(n-r)}
\to 0.
\end{equation}
In this setting, we have the following cohomological criterion due
to Narasimhan and Ramanan.
\begin{lemma} \label{cohomology} \ {\rm (\cite{NaRa}, Lemma 3.1)  }
The elementary transformation $\widetilde{F} \subset F$ lifts to a
map $\widetilde{F} \to V$ so that the diagram (\ref{lifting}) is
commutative if and only if $v \in \pp = \pp H^1(X, Hom(F,E))$ lies
in $\pp (\ker \beta)$ for the map
\begin{equation} \label{beta}
\beta: H^1(X, Hom(F,E)) \to H^1 (X, Hom(\widetilde{F},E))
\end{equation}
associated to the exact sequence (\ref{Elem}). \qed
\end{lemma}
%\begin{proof}
%Consider the following diagram.
%\begin{equation} \label{diag}
%\begin{CD}
%\Hom(\widetilde{F}, V) @>\epsilon>> \Hom(\widetilde{F}, F)     @>\delta>>    H^1 (X, Hom(\widetilde{F},E))  \\
%@.                 @A{\alpha}AA                       @A{\beta}AA   \\
%@.              \Hom(F,F)     @>{\partial}>>    H^1 (X, Hom(F,E))   \\
%\end{CD}
%\end{equation}
%Here the horizontal arrows are obtained from the original extension
%(\ref{original}) and the vertical arrows are from (\ref{dual}) and
%(\ref{Elem}) respectively. By functoriality, this diagram is
%commutative.
%
%Recall that $v \in \pp$ is given by $[\partial (1_F)]$ for the
%identity element $1_F \in \Hom (F,F)$. Hence there is a lifting
%$\widetilde{F} \to V$ if and only if $\alpha(1_F)$ lies in the image
%of $\epsilon$, and this is equivalent to that $ v \in \pp (\ker
%\beta)$.
%\end{proof}
For further discussions, it will be convenient to have an explicitly
constructed parameter space of the elementary transformations of
$F$.
\begin{lemma} \label{quot}
For a vector bundle $F$ and a fixed number $k>0$, the elementary
transformations $\widetilde{F} \hookrightarrow F$ with
$\deg(F/\widetilde{F}) = k$ are parameterized by a projective
variety $Q_k$ of dimension at most $k \cdot \rk (F)$.
\end{lemma}
\begin{proof}
One can take the Quot scheme of $F$ parameterizing the surjections
$F \to \tau$, where $\tau$ runs over the space of torsion sheaves of
degree $k$. A more explicit parameter space can be constructed as
follows.

If $k=1$ then $Q_1 = \pp F^*$: Any linear functional $\mu : F_x \to
\cc$ gives an elementary transformation $\widetilde{F}$ whose
quotient $F/\widetilde{F}$ has degree 1. Conversely, any elementary
transformation $\widetilde{F}$ with degree 1 quotient determines an
element of $F^*$ up to a constant. Furthermore, there is a bundle
$p_1: \cF_1 \to Q_1 \times X$ such that for each $\mu \in Q_1$, the
bundle $\widetilde{F} = \cF_1|_{[\mu] \times X}$ is the kernel of $F
\stackrel{\mu}{\to} \cc_x$.
%the composition $F \to F|_x \stackrel{\mu}{\to} \cc$.
Indeed, let $\pi_X: Q_1 \to X$ be the projection and let $\bar{Q}_1$
be the copy of $Q_1$ embedded in $Q_1 \times X$ as the graph of
$\pi_X$. Also let $\pi_{Q_1}:Q_1 \times X \to Q_1$  be the
projection. Then the bundle $\cF_1$ is defined by the following
sequence over $Q_1 \times X$:
\[
0 \to \cF_1 \to \pi_X^* F \otimes \pi_{Q_1}^* \cO(-1) \to
\cO_{\bar{Q}_1} \to 0.
\]
Here the quotient map is given by the composition
\[
\pi_X^* F \otimes \pi_{Q_1}^* \cO(-1) \ \to \ \pi_X^* F \otimes
\pi_{Q_1}^* \cO(-1) \otimes \cO_{\bar{Q}_1} \ \to \ \cO_{\bar{Q}_1}
\]
where the second map is the pairing $f \otimes \mu \mapsto \mu(f)$
for $f \in F_x$ and $\mu \in F_x^*$ for some $x \in X$.

By induction, assume that there is a variety $Q_k$ parameterizing
the elementary transformations of $F$ with $\deg (F/ \widetilde{F})
= k$, together with a classifying bundle $p_k: \cF_k \to Q_k \times
X$. Each $\mu \in \pp \cF_k^*$ represents a bundle $F_k$ in $Q_k$, a
point $x \in X$ and a codimension 1 subspace of $F_x$, or
equivalently an elementary transformation $ \tilde{F_k} = \ker [F_k
\to \cc_x]$. Hence the space $Q_{k+1} := \pp \cF_k^*$ parameterizes
the elementary transformations $\widetilde{F}$ of $F$ with quotient
of degree $k+1$.

Also, we have a classifying bundle $\cF_{k+1}$ over $Q_{k+1} \times
X$: let $q^{k+1}_k$ and $q^{k+1}_X$ be the compositions $Q_{k+1} \to
Q_k \times X \to Q_k$ and $Q_{k+1} \to Q_k \times X \to X$
respectively. Let $\bar{Q}_{k+1}$ be the copy of $Q_{k+1}$ embedded
in $Q_{k+1} \times X$ as the graph of $q_X^{k+1}$. Let
$\pi_{Q_{k+1}}: Q_{k+1} \times X \to Q_{k+1}$ be the projection.
Then the bundles $\cF_{k+1}$ is defined by the sequence over
$Q_{k+1} \times X$:
\[
0 \to \cF_{k+1} \to (q_k^{k+1} \times id_X)^*\cF_k \otimes
\pi_{Q_{k+1}}^*\cO(-1) \to \cO_{\bar{Q}_{k+1}} \to 0.
\]
In this construction, the parameter space is given by a tower of
projective bundles starting from $Q_1 = \pp F^* \to X$, which is a
projective variety. Moreover, at each step the dimension of $Q_k$
increases by $\rk (F)$, so $\dim Q_k = k \cdot \rk(F) $.
\end{proof}
\begin{lemma} \label{general} \ \
Assume $E \in U(n-r,0)$ and $F \in U(r,d)$ where $d > (2g-1)r$, so
that by Lemma \ref{embedding}, the projective bundle $\pp Hom(F,E)$
is embedded in $\pp$ as a scroll. Consider the diagram
(\ref{lifting}) and the associated map (\ref{beta}). For a general
$\widetilde{F}$ in $Q_k$:
\begin{enumerate}
\renewcommand{\labelenumi}{(\roman{enumi})}
\item The support of $\tau = F / \widetilde{F}$ consists of $k$
distinct points $x_1, x_2, \ldots, x_k$ of $X$. Moreover,
$\widetilde{F}$ is the intersection of the subsheaves $\ker \left[
\mu_i: F \to \cc_{x_i} \right]$ for some $\mu_i \in F_{x_i}^*$, $i =
1, 2, \ldots, k$.
\item The subspace $ \pp (\ker \beta) \in \pp$ is the join of the $k$
distinct linear spaces
\[ \pp (\mu_i \otimes E_{x_i}) \subset \pp Hom(F,E)_{x_i} \]
for $i=1,2, \ldots, k$. \end{enumerate}
\end{lemma}
\begin{proof}
(i) In the description of $Q_k$ in Lemma \ref{quot}, the support of
$\tau$ of $\widetilde{F} \in Q_k$ is given by the images $x_1, x_2,
\cdots, x_k$ of $\widetilde{F}$ under the projections
\[
\pi_1 = q_X^k, \quad \pi_2 = q_X^{k-1} q_{k-1}^k, \quad \ldots,
\quad \pi_k= \pi_X q_1^2 \ldots q^{k-1}_{k-2} q_{k-1}^k.
\]
Let $D$ be the big diagonal inside $X^k$ consisting of the
$k$-tuples $(x_1, x_2, \ldots, x_k)$ such that $x_i = x_j$ for some
$1 \le i \neq j \le k$. The inverse image of $X^k \setminus D$ under
the map $(\pi_1, \pi_2, \cdots, \pi_k): Q_k \to X^k$ parameterizes
those $\widetilde{F}$ such that the support of $\tau$ consists of
$k$ distinct points.
\par
(ii) The description of $\pp (\ker \beta) \in \pp$ as a join of $k$
distinct linear subspaces is a direct consequence of (i) and the
definition of the embedding $\phi: \pp Hom(F,E) \to \pp$, as was
explained in Section 3.
\end{proof}
Now we find a geometric criterion for the lifting of $\widetilde{F}$
as in (\ref{lifting}).
\begin{theorem} \label{geometry} \ \
Assume $E \in U(n-r,0), F \in U(r,d)$ where $d > (2g-1)r$ so that
$\Delta$ is a subvariety of the scroll $\pp Hom(F,E)$ in $\pp$.
Consider the diagram (\ref{lifting}) and the associated map
(\ref{beta}).
\begin{enumerate}
\renewcommand{\labelenumi}{(\roman{enumi})}
\item  If some elementary transformation $\widetilde{F} \in Q_k$ lifts
to $V$ so that the diagram (\ref{lifting}) is commutative, then $v
\in Sec^k \Delta$, where $\Delta$ is the sublocus of the scroll $\pp
Hom(F,E)$ defined in \ref{rank1}.\\
\item If a point $v \in Sec^k \Delta$ is general, then $V$ admits an
elementary transformation $\widetilde{F} \in Q_k$ as a subsheaf.\\
\item Either the condition that $V$ admits an elementary
transformation $\widetilde{F} \in Q_k$ as a subsheaf or the
condition $v \in Sec^k\Delta$ implies that $s_r(V) \le rd - n(d-k)$.
\end{enumerate}
\end{theorem}
\begin{proof} \
If an elementary transformation $\widetilde{F} \subset F$ has a
lifting $\widetilde{F} \to V$, then $v \in \pp (\ker \beta)$ by
Lemma \ref{cohomology}. If $\widetilde{F} \in Q_k$ is general, then
by Lemma \ref{general}, the point $v$ lies on the linear space
spanned by $[\mu_i \otimes e_i] \in \Delta$ for some $\mu_i \in
F_{x_i}^*$ and $e_i \in E_{x_i}$, $1 \le i \le k$. Therefore $v \in
Sec^k \Delta$. By continuity, we get (i).

Now suppose that $v$ is a general point of $Sec^k \Delta$. Then $v$
lies on some linear space spanned by $\{ [\mu_i \otimes e_i] \in
\Delta_{x_i} : i = 1,2,\cdots, k \}$ for some distinct points $x_1,
x_2, \ldots, x_k \in X$. This data defines an elementary
transformation $\widetilde{F}$ of $F$, which is general if we choose
the vectors $\mu_i \otimes e_i$ generally. Hence $v \in \pp (\ker
\beta)$ by Lemma \ref{general} and $\widetilde{F} \to F$ has a
desired lifting $\widetilde{F} \to V$ by Lemma \ref{cohomology}.
This proves (ii).

Finally, consider (iii). The condition that $V$ admits an elementary
transformation $\widetilde{F} \in Q_k$ as a subsheaf, obviously
implies $s_r(V) \le rd - n(d-k)$. Now consider the family of vector
bundles parameterized by $Sec^k \Delta \subset \pp$. From (ii), a
general bundle $V$ with $v \in Sec^k \Delta$ has a subsheaf
$\widetilde{F}$ of degree $d-k$, hence $s_r(V) \le rd - n(d-k)$.
Since the invariant $s_r$ is lower semicontinuous, this inequality
holds for every bundle $V$ with $v \in Sec^k \Delta$.
\end{proof}
\begin{remit}
{\rm (i) It can be asked if the inequality $s_r(V) \le rd - n(d-k)$
implies either the condition $v \in Sec^k \Delta$ or the lifting of
some $\widetilde{F} \in Q_k$ to $V$. The answer seems to be No, due
to the possible existence of diagrams of the following form:
\begin{equation} \label{exception}
\begin{CD}
0 @>>> E  @>>>  V  @>>>  F_d   @>>>  0   \\
@.   @AAA      @AAA         @AAA               @.  \\
0  @>>>    E'                @>>>  F_{d-k}  @>>>  E''  @>>> 0
\end{CD}
\end{equation}
The point is that a subsheaf $F_{d-k}$ of $V$ of rank $r$ and degree
$d-k$ may intersect $E$ in some subsheaf $E'$ of rank $\ge 1$. If
one can prove that this kind of diagram does not exist for general
$V$, then we get the equivalence
\begin{equation} \label{equivalence}
s_r(V) \le rd - n(d-k) \ \Leftrightarrow \ v \in Sec^k \Delta
\end{equation}
by the closedness of both spaces. But in general, we cannot exclude
the possibility that every maximal subbundle $F_{d-k}$ of $V$ yields
a diagram of the form (\ref{exception}).

In the special case when either $r=1$ or $r=n-1$, the diagram
(\ref{exception}) does not exist, and we get the equivalence
(\ref{equivalence}). When $n=2$ and $r=1$, this is exactly the
statement in Proposition \ref{LN}.\\
(ii) Also one may ask if the generality assumption in (ii) can be
dropped. The answer is Yes if $\dim (\ker \beta)$ is constant for
every $\widetilde{F} \in Q_k$. In this case, there is a projective
bundle $\cP$ over $Q_k$ whose fiber over $\widetilde{F}$ is the
corresponding $\pp (\ker \beta)$. This fibration induces a morphism
$\cP \to \pp$ which maps onto $Sec^k \Delta$. But in general, $\dim
(\ker \beta)$ may drop, since there may exist $\widetilde{F} \in
Q_k$ with $h^0(X, Hom(\widetilde{F}, E)) \neq 0$. If $k<d$ then
$\deg Hom(\widetilde{F}, E)<0$, so such an $\widetilde{F}$ is
necessarily unstable.}
\end{remit}

\section{Hirschowitz bound}

In this section, we reprove Hirschowitz bound in Proposition
\ref{Hirschowitz} by applying our geometric criterion on lifting of
elementary transformations.

Let $V$ be a bundle in $U(n,d)$. Note that the upper bound of $s_r$
is attained by a general bundle in $U(n,d)$, due to the
semi-continuity of $s_r$. Hence it suffices to prove the upper bound
for general $V$. Moreover, since $s_r(V) = s_r(V \otimes L)$ for any
line bundle $L$, we may assume that $d = \deg V
> (2g-1)n$. Then by Lemma \ref{setting}, $V$ fits into the exact
sequence
\[
0 \to E \to V \to F \to 0
\]
for some general bundles $E \in U(n-r,0)$ and $F \in U(r,d)$. Let $v
\in \pp = \pp H^1(X, Hom(F,E))$ be the point corresponding to this
sequence.

Consider the subvariety $\Delta$ of the scroll $\pp Hom(F,E) \subset
\pp$. Since $\dim \Delta = n-1$, the expected dimension of $Sec^k
\Delta$ is $nk-1$ unless $Sec^k \Delta= \pp$. Note that $\dim \pp =
(n-r)(d + r(g-1))-1$ by Riemann-Roch formula. Thus if $\Delta
\subset \pp$ has no secant defect, then
\[ v \in Sec^m \Delta \ \
\text{for } m = \left\lceil \frac{n-r}{n} \left( d+r(g-1) \right)
\right\rceil.
\]
We may write $(n-r)(d+r(g-1)) = mn - \varepsilon$ for some
$\varepsilon$ in the range $0 \le  \varepsilon \le n-1$. By Theorem
\ref{geometry} (iii),
\begin{eqnarray*}
s_r(V) &\le& rd - n(d-m) \\
&=& r(n-r)(g-1) + \varepsilon.
\end{eqnarray*}
Therefore the Hirschowitz bound is obtained from the following.
\begin{theorem} \label{defect}
For general bundles $E \in U(n-r,0)$ and $F \in U(r,d)$ where $d >
(2g-1)n$, the subvariety $\Delta$ in $\pp$ has no secant defect.
\end{theorem}
To prove this, we invoke the Terracini lemma (cf.\ \cite{Te}).
\begin{proposition}
(Terracini lemma) \ \ Let $Z \subset \pp^N$ be a projective variety,
and let $z_1, \ldots, z_k$ be general points of $Z$. Then
\[
\dim (Sec^k Z) = \dim <\TT_{z_1} Z, \ldots, \TT_{z_k} Z >,
\]
where $\TT_{z_i}Z$ denotes the embedded tangent space to $Z$ in
$\pp^N$ at $z_i$. \qed
\end{proposition}
Now we use the dictionary in Section 3 to describe the embedded
tangent spaces of $\Delta \subset \pp Hom(F,E)$ in terms of the
elementary transformations.
\begin{lemma} \label{tangent}
Let $[\mu \otimes e] \in \Delta_x$ in $\pp Hom(F,E)|_x$. Let
$\widehat{F^*}$ be the elementary transformation of $F^*$ at $\mu$.
Also let $\widehat{E}$ be the elementary transformation of $E$ at
$e$. Consider the elementary transformation
\begin{equation} \label{embedded}
0 \to F^* \otimes E \to \widehat{F^*}  \otimes \widehat{E} \to \tau
\to 0
\end{equation}
where $\tau$ is some torsion sheaf $\tau$ of degree $n$. Then
\[
\TT_{[\mu \otimes e]} \Delta = \pp \ker [\gamma: H^1(X, F^* \otimes
E) \to H^1(X, \widehat{F^*}  \otimes \widehat{E})]
\]
for the map $\gamma$ coming from the long exact sequence associated
with (\ref{embedded}).
\end{lemma}
\begin{proof}
Since we are considering a local problem, we consider
trivialisations of $F^*$ and $E$ over a suitable neighbourhood $U$
of $x$. Then $F^* \otimes E|_U \cong \cO_U^{\oplus r(n-r)}$. Taking
elementary transformations, we obtain
\[
\widehat{F^*} \otimes \widehat{E} \ \cong \ \cO_U(2x) \oplus
\cO_U(x)^{\oplus (n-2)} \oplus \cO_U^{\oplus (r-1)(n-r-1)}.
\]
Therefore the skyscraper sheaf $\tau$ is isomorphic to $\cc_{2x}
\oplus \cc_{x}^{\oplus (n-2)}$. The first term $\cc_{2x}$
corresponds to the line defined by the constant section $\mu \otimes
e : U \to \pp(F^* \otimes E)$. The next term $\cc_{x}^{\oplus
(n-2)}$ together with the subsheaf $\cc_{x} \subset \cc_{2x}$
correspond to the join of two linear spaces
\[ \pp (\mu \otimes E_x) \cong \pp^{n-r-1} \quad \hbox{and} \quad \pp (F^*_x \otimes e) \cong \pp^{r-1} \]
which intersect at $[\mu \otimes e]$. Since all of these are tangent
data of $\Delta$ and they span a linear space of dimension $n-1$, we
get the desired result.
%The first term $\cO_U(2x)$ corresponds to the line defined by the
%constant section $\mu \otimes e : U \to \pp(F^* \otimes E)$. The
%next term $\cO_U(x)^{\oplus (n-2)}$ together with the subsheaf
%$\cO_U(x) \subset \cO_U(2x)$ correspond to the join of two linear
%spaces
%\[ \pp (\mu \otimes E_x) \cong \pp^{n-r-1} \quad \hbox{and} \quad \pp (F^*_x \otimes e) \cong \pp^{r-1} \]
%which intersect at $[\mu \otimes e]$. Since all of these are some tangent data of $\Delta$ and they
%span a linear space of dimension $n-1$, we get the desired result.
\end{proof}
\noindent {\it Proof of Theorem \ref{defect}}. \ From the dimension
computations in the beginning of this section, it suffices to show
that
\[
\dim Sec^{k} \Delta = nk-1
\]
for every $k<m = \left\lceil \frac{n-r}{n} \left( d+r(g-1) \right)
\right\rceil$, and
\[
\dim Sec^m \Delta = \dim \pp = nm-1- \varepsilon,
\]
where $0 \le \varepsilon \le n-1$. By the Terracini lemma,
\[
\dim Sec^k \Delta \ = \ \dim \left( < \TT_{[\mu_1 \otimes e_1]}
\Delta, \ldots, \TT_{[\mu_k \otimes e_k]} \Delta > \right)
\]
for general points $[\mu_i \otimes e_i] \in \Delta, 1 \le i \le k$.

By generality, we may assume that $\mu_1 \otimes e_1, \ldots, \mu_k
\otimes e_k$ lie over $k$ distinct points of $X$. In this case, by
Lemma \ref{tangent}, the join of the spaces $ \TT_{[\mu_i \otimes
e_i]} \Delta$ in $\pp$ can be expressed as the projectivised kernel
of the map
\[
\Gamma: H^1(X, F^* \otimes E) \to H^1(X, \widehat{F^*} \otimes
\widehat{E})
\]
where $\widehat{F^*}$ (resp. $\widehat{E}$) are the elementary
transformations of $F^*$ (resp. $E$) at $\mu_1, \ldots, \mu_k$
(resp. $e_1, \ldots, e_k$).

Note that
\begin{eqnarray*}
\deg (\widehat{F^*} \otimes \widehat{E}) &=& r \deg \widehat{E} +
(n-r) \deg \widehat{F^*} = rk + (n-r) (-d +k)\\
&=& r m+(n-r)(-d+m) - n(m-k) \\
&=& nm   - (n-r)d - n(m-k)\\
 &=& r(n-r)(g-1) + \varepsilon - n(m-k),
\end{eqnarray*}
where $0 \le \varepsilon \le n-1$. If $\widehat{F^*} \otimes
\widehat{E}$ is nonspecial, then $h^0(X, \widehat{F^*} \otimes
\widehat{E}) = 0$ for $k<m$ and $h^0(X, \widehat{F^*} \otimes
\widehat{E}) = \varepsilon$ for $k= m$. Thus
\begin{eqnarray*}
\dim \pp (\ker \Gamma) &=& - \deg (F^* \otimes E) + \deg
(\widehat{F^*} \otimes \widehat{E}) - h^0(X, \widehat{F^*} \otimes
\widehat{E}) -1\\
&=& n k - h^0(X, \widehat{F^*} \otimes \widehat{E}) -1 \\
&=& \min \{ nk-1, \dim \pp \}
\end{eqnarray*}
as expected. Therefore it remains to check that $\widehat{F^*}
\otimes \widehat{E}$ is nonspecial. But we assumed that $E$ and $F$
are general, and $\widehat{E}$ and $\widehat{F^*}$ are obtained from
a general elementary transformation of $E$ and $F^*$ respectively,
hence they are general. By Hirschowitz' lemma (\cite{Hi} 4.6, see
also \cite{RuTe} Theorem 1.2), the tensor product of two general
bundles is nonspecial. \qed\\

{\it Acknowledgements}. \ \  The authors would like to thank Peter
Newstead for his kind answer regarding Lemma \ref{setting}. The
first author acknowledges the hospitality of H\o gskolen i Vestfold
during his visit in January 2008. The second named author
acknowledges gratefully the hospitality of the Korean Institute for
Advanced Study and Konkuk University during his visits in October
2007 and February 2009 respectively.

\end{document}